\documentclass{amsart}

\usepackage[T1]{fontenc}
\usepackage{amsfonts,amsmath,amsthm,amsbsy,wasysym,latexsym,amssymb,graphicx,booktabs,fancybox,color,multicol,manfnt,fullpage,lscape,geometry,enumitem,tikz}
\usetikzlibrary{cd}
\usepackage{hyperref}
\usepackage{appendix,spverbatim,colonequals,mathtools,stmaryrd}
\usepackage{listings}

\lstdefinelanguage{Sage}[]{Python}
{morekeywords={False,sage,True},sensitive=true}
\definecolor{dblackcolor}{rgb}{0.0,0.0,0.0}
\definecolor{dbluecolor}{rgb}{0.01,0.02,0.7}
\definecolor{dgreencolor}{rgb}{0.2,0.4,0.0}
\definecolor{dgraycolor}{rgb}{0.30,0.3,0.30}

\newcommand{\Z}{\mathbf{Z}}
\newcommand{\F}{\mathbf{F}}
\newcommand{\Q}{\mathbf{Q}}
\newcommand{\C}{\mathcal{C}}
\newcommand{\D}{\textbf{D}}
\newcommand{\p}{\bar{\rho}}
\newcommand{\rps}{R^{\mathrm{ps}}}
\newcommand{\rtr}{R^{\mathrm{coeff}}}
\newcommand{\m}{\mathfrak{m}}
\def\Runiv{R^{\mathrm{univ}}}
\def\univ{\mathrm{univ}}

\def\coeff{\mathrm{coeff}}
\def\D{\mathbf{D}}

\newcommand{\GL}{{\rm GL}}

\newtheorem{thm}{Theorem}[section]
\newtheorem{lemma}[thm]{Lemma}
\newtheorem{prop}[thm]{Proposition}

\newtheorem{ithm}{Theorem}

\theoremstyle{definition}

\newtheorem{eg}{Example}[section]
\newtheorem{defi}[thm]{Definition}
\newtheorem{rmk}[thm]{Remark}

\title{Pseudorepresentations Not Arising from genuine representations}
\author{Jinyue Luo}
\address{Department of Mathematics, The University of Chicago, 
	Chicago, IL 60637}
\email{jinyue@math.uchicago.edu}
\date{}

\begin{document}
\maketitle

\begin{abstract}
We show that a pseudorepresentation $\D\colon A[G] \rightarrow A$ of a (finite) group $G$ need not arise from a genuine  representation, even if one is allowed
to extend the ring~$A$.  This shows that a theorem of the "embedding problem" for residually multiplicity free pseudorepresentations ~\cite[Proposition~1.3.13]{8} can not be extended to the general setting. 
\end{abstract}

\section{Introduction}
In~\cite{6}, Mazur initiated the study of deformations of residual Galois representations:
$$\p\colon G \rightarrow \GL_2(k)$$
for a finite field~$k$. Mazur showed that when~$\p$ was absolutely irreducible, there exists a universal deformation ring $\Runiv$.
In applications to modularity, one often identifies a certain restricted deformation ring~$R$ of~$\p$ with a (localized) Hecke algebra~$\mathbf{T}$, where the Hecke operators~$T_{v}$ for
general primes~$v$ are identified with the traces of Frobenius of~$\p: G_{\Q} \rightarrow \GL_2(R)$.

If~$\p$ is reducible, however, then the universal deformation ring~$\Runiv$ does not always exist. One can sometimes get around this issue by instead considering the universal 
\emph{framed} deformation ring~$R^{\square,\mathrm{univ}}$. However, the relationship between~$R^{\square,\mathrm{univ}}$ (and various restricted deformation rings~$R^{\square}$) and Hecke rings~$\mathbf{T}$ is now more subtle.

One approach is to consider the subring of~$R^{\square,\univ}$ generated by the traces of images of~$g \in G$ and to compare that to~$\mathbf{T}$. A second approach is to replace representations of~$G$
by \emph{pseudorepresentations}.
Pseudorepresentations of a group $G$ over a commutative ring $A$ are $A$-valued functions on $G$ that behave like characters of finite-dimensional representations. In this context, pseudo-representations were first considered by Wiles in~\cite{2} where they were called pseudo-characters. 
In Wiles' setting, the dimension $d$ is $2$. Then Taylor generalized to any dimension $d$ with the assumption $d! \in A^*$ ~\cite{3}. In this paper, we will mainly consider the pseudorepresentations of Chenevier ~\cite{1} called "determinants" which in particular allow us to deal with the case where $d!$ is not invertible in $A^*$ (see Sec \ref{pseudo} for precise definitions). From Proposition E of ~\cite{1}, the deforming functor is prorepresentable by the universal pseudodeformation ring $\rps$.

Given a pseudorepresentation $\D$ of~$G$ valued in a ring~$A$, it is natural to ask whether~$\D$ arises from a genuine representation. This is the "embedding problem".
One can ask this question in two forms:
\begin{enumerate}
\item Does there exist a representation~$\rho\colon G \rightarrow \GL(A)$ whose associated pseudorepresentation is~$\D$?
\item Does there exist a ring~$B$, an injective morphism~$A \rightarrow B$, and a representation~$\rho\colon G \rightarrow \GL(B)$ whose associated pseudorepresentation in~$B$ is valued in~$A$ and coincides with~$\D$?
\end{enumerate}
The following is the main theorem of the paper. 
\begin{ithm}[Theorem \ref{extraspecial}]
There exists a finite group~$G$ and a pseudorepresentation $\D$ of~$G$ serving as a counter-example to the questions above. 
\end{ithm}

The second version of this question is the weaker version and is the one we consider here. (Any counter-example to the second question immediately gives a counter-example to the first question.) %Therefore the counter-example we show in Section 3 will give negative answer to the both versions of this question.

Given a genuine representation $\rho \colon G \to \GL_d(A)$, extend it to $\rho \colon G \to M_d(A)$ by linear extension and define $\D_{\rho}\coloneqq \det\circ\rho$ which is a degree $d$ determinant from $A[G]$ to $A$. It is a $d$-dimensional pseudorepresentation by definitions in \ref{pseudo}.

\begin{defi}
    If $B$ is an $A$-algebra and $m \in B[G]$, we could extend the action of determinant to $B$ and define the the characteristic polynomial $P(\D,m)(X) \in B[X]$ of $m$ by the formula
    $$P(\D,m)(X)\colonequals \D(X-m) \eqqcolon \sum_{i=0}^d (-1)^i\Lambda_i(m)X^{d-i}.$$
\end{defi}

\begin{lemma}{~\cite[Corollary~7.11]{1}}
    Let $\D$ be an $A$-valued determinant on $G$ of dimension $d$ and $R \subset A$ the subring generated by the coefficients $\Lambda_i(g)$ of $P(\D,g)(X)$ for all $g\in G$. Then $\D$ factors through a (unique) $R$-valued determinant on $G$ of dimension $d$.
\end{lemma}

Fix a residual representation $\p\colon G \to\GL_d(k)$ and consider the universal lifting $\rho^{\univ}$, which induces a $d$-dimensional pseudorepresentation $\D_{\rho^{\univ}}$. Let $\rtr_{\p} \subset R_{\p}^{\square,\univ}$ be the subring generated by the coefficients of characteristic polynomials of $\rho(g)$ for all $g \in G$. From the previous lemma we deduce that $\D_{\rho^{\univ}}$ factors through $\rtr_{\p}$. 

\begin{lemma}
 Let $\rps_{\p}$ be the universal pseudodeformation ring of the residual pseudorepresentation $\D_{\p}$ (see Sec \ref{pseudo} for precise definitions). Then there exists a natural surjection $\rps_{\p} \twoheadrightarrow \rtr_{\p}$. And if the answer to the second version of the above question is true, this surjection will be an isomorphism. We omit the subcript $\p$ for simplicity.
\end{lemma}

Therefore any counter-example such that $\rps \ncong \rtr$ will give a counter-example to the above question. Now we begin the proof of the lemma.
\begin{proof}
Let $\D^{\univ}$ denote the universal pseudorepresentation.
By the universal property, there is a natural map $\rps \to \rtr$ which brings $\D^{\univ}$ to $\D_{\rho^{\univ}}$. Each generator $\Lambda_i(g)$ of $\rtr$ is in the image. In other words, it is a surjection.

Now suppose for any pseudorepresentation $\D \colon A[G] \to A$, there exists a ring $B$ with an injection $A \hookrightarrow B$, and a representation $\rho \colon G \to \GL(B)$ satisfying $\D_{\rho}$ is valued in $A$. Take $A$ to be $\rps$ and $\D$ to be $\D^{\univ}$. Consider the corresponding $B$ and $\rho$. By our assumption $\D_{\rho}$ coinsides with $\D^{\univ}$. By the universal property, there is a natural map $R^{\square,\univ} \to B$ taking $\rho^{\univ}$ to $\rho$, which also induces $\D_{\rho^{\univ}} \to \D_{\rho}$. Let $B^{\coeff}$ denote the subring generated by the coefficients of characteristic polynomials of $\rho(g)$. Similarly as above we will have the surjection $\rtr \to B^{\coeff}$ which brings $\D_{\rho^{\univ}}$ to $\D_{\rho}=\D^{\univ}$. Applying the universal property of pseudodeformations again, we conclude that $\rtr \cong \rps$.
\end{proof}

For Taylor-style pseudorepresentations (with the assumption $d! \in A^*$), Taylor showed in ~\cite{3}, relying on earlier results by Procesi~\cite{17}, that the answer to the previous question (the second version) is always yes in the case where $A$ is an algebraically closed field of characteristic zero; this result was extended, with a different method, to any algebraically closed field by Rouquier. The question was settled in 1996 for any local henselian ring $A$, in the case where the residual pseudocharacter $\bar{T}$ is absolutely irreducible, independently by Rouquier ~\cite{9} and Nyssen ~\cite{10}.

More generally, if $\bar{T}\colon A[G] \to A$ is multiplicity-free, then in ~\cite[Section~1.4]{8} it was shown every Cayley-Hamilton quotient of $A[G]$ is a GMA (\textit{generalized matrix algebra}). And any GMA over $A$ can be embedded, compatibly with  the trace function, in an algebra $M_d(B)$ for some $A$-algebra $B$ ~\cite[Section~1.3]{8}. As mentioned in ~\cite[P.~1630]{4}, all these arguments of ~\cite{8} have no dependence on the characteristic of $A$ or the invertibility of $d!$ in $A$. Therefore $\rps$ is precisely the GIT quotient ring $R=\Gamma(\mathcal{O}_{D_{\p}^{\square}})^{\text{GL}_d}$ ~\cite[Theorem~A(3)]{4}.

This paper deals with the case where $p=d=2$ and the residual pseudorepresentation is trivial. Our main result is that the map $\rps\twoheadrightarrow \rtr$ is \emph{not} an isomorphism, even for finite groups~$G$.

\begin{thm}[Theorem \ref{extraspecial}]
For the extraspecial group $G=2_+^{1+4}, \rps \ncong \rtr$. 
\end{thm}

The main idea of our proof is to explain how to explicitly compute (enough of) both sides to prove that these rings are different.
\subsection*{Notation and Conventions}
$k$ is a finite field of characteristic $p$. $G$ is a topological group satisfying the Mazur's finiteness condition $\Phi_p$ in \cite{6}. $d$ is a positive integer.

\section{Different Galois Deformation Functors}
\subsection{Deformation functors associated to $\p$}\label{deformingfunctor}

Let $\p \colon G\to \GL_d(k)$ be a continuous representation. Let $\C$ be the category of local Artinian $W(k)$-algebras $(A,\m_A)$ with a fixed identification $A/\m_A \cong k$. Let $D_{\p}^{\square} \colon \C \to \text{Sets}$ be the functor such that for $(A,\m_A) \in \C$, $D_{\p}^{\square}$ is the set of continuous representations $\rho_A \colon G \to \text{GL}_d(k)$, such that $\rho_A (\text{mod\ } \m_A) = \p$. We call such a $\rho_A$ a lifting of $\p$. The functor $D_{\p}^{\square}$ is pro-represented by a complete local Noetherian $W(k)$-algebra $R_{\p}^{\square}$ and we will have the universal lifting $\rho^{\univ}\colon G\to \text{GL}_d(R_{\p}^{\square})$ by Schlessinger's criterion. 

\begin{eg}\label{Example}
Suppose we have a trivial residual representation $\p\colon G\to\text{GL}_d(k)$ for a finitely presented group $G=\langle g_s|f_t\rangle_{1\le s\le n,1\le t\le m}$, where $g_s$ denote the generators and $f_t(g_1,\dots,g_n)$ denote the relations of the generators. Then a lifting $\rho$ to any $A$ is specified by the images of $g_s$, subject to the relations 
$$f_j(\rho(g_1),\dots,\rho(g_n))=0.$$
Let $[I_d+(x_s^{ij})]$ denote the matrix form of $\rho(g_s)$.
So we could take $$R^{\square}_{\p}=W(k)\llbracket\{x_s^{ij}\}_{1\le s\le n,1\le i,j\le d}\rrbracket/I$$
where the ideal of relations $I$ is generated by the ones given by the matrix equations
$$f_j([I_d+(x_1^{ij})],\dots,[I_d+(x_n^{ij})])=0.$$
\end{eg}
\begin{defi}
We define the \textit{coefficient subring} $\rtr_{\p}$ to be the subring of $R^{\square}_{\p}$ generated by the coefficients of the characteristic polynomials of the images of $g \in G$.
\end{defi}
\begin{rmk}
If the dimension $d$ is 2, $\rtr$ is generated by all the traces $\text{tr}(\rho^{\univ}(g))$ and determinants $\text{det}(\rho^{\univ}(g))$ for all $g \in G$. And if the the characteristic of $k$ is not 2, then $\rtr$ is generated by the traces of images of~$g \in G$ since $\mathrm{det}(g)=\frac{\mathrm{tr}(g)^2-\mathrm{tr}(g^2)}{2}$. That is why it was sometimes denoted by the "trace subring".
\end{rmk}

\subsection{Definition of Pseudorepresentations}\label{pseudo}
In this paper, we will mainly consider the pseudo-representations introduced by Chenevier~\cite{1} called "determinants". These in particular allow us to deal with cases where $p < d$. We will mainly be concerned with determinants of degree $d=2$ and $p=2$ in this paper. Let $G$ be a group, $A$ be a commutative unital ring, and let $M$ and $N$ be two $A$-modules. Let $\C_A$ be the category of commutative $A$-algebras. Each $A$-module $M$ gives rise to a functor $\underline{M}\colon \C_A \to \text{Sets}$ via the formula $B\mapsto M\otimes_A B$. 

\begin{defi}[Polynomial Law]
An $A$-valued polynomial law $\D$ between two $A$-modules $M$ and $N$ is by definition a natural transformation $\underline{M} \to \underline{N}$. In other words, it is a collection of maps $\D_B \colon M\otimes_AB\to N\otimes_AB$, where $B$ is any commutative $A$-algebra.
And one can extend the action of $A[G]$ on $M$ to an action of $B[G]$ on $M\otimes_AB$.
A polynomial law is called
\begin{enumerate}
    \item \textit{multiplicative}, if $\D(1)=1$ and $\D(xy)=\D(x)\D(y)$ for all $x,y \in A[G]\otimes B$, \item \textit{homogenenous of degree} $d$, if $\D(xb)=b^d\D(x)$ for all $x\in A[G]\otimes B$ and $b \in B$,
    \item \textit{continuous}, if its characteristic polynomial map on $G$ given by $g \mapsto P(\D,g)$ is continuous.
\end{enumerate}
\end{defi}
\begin{defi}
A degree $d$ determinant is a continuous $A$-valued polynomial law $\D\colon A[G]\to A$, which is multiplicative and homogeneous of degree $d$. 
\end{defi}
We follow the discussion of continuous determinants in ~\cite[Section~2.5]{1} and equip $A$, as well as all the commutative $A$-algebras $B$, with the discrete topology.
In the computation in this paper, we mainly consider the simplest examples, finite groups $G$ with discrete topology.
All representations and all determinants considered in this paper are continuous.

\begin{prop}{~\cite[Proposition 7.4]{1}}
The covariant functor $\det_A(A[G],d)$ associating to any commutative $A$-algebra $B$, the set of $B$-valued determinants $A[G]\otimes_A B \to B$ of dimension $d$, is representable by $\Gamma_A^d(A[G])^{ab}$ (the abelianization of divided powers of order $d$, its detailed definition could be found in ~\cite{1}).
\end{prop}
The "embedding problem" is local on $\mathrm{Spec}(\Gamma_{\Z}^d(\Z[G])^{ab})$ (\cite[Remark 7.20]{1}). Now localize this question and consider the pseudodeformation functor $F$. Given a residual determinant $\bar{\D}\colon k[G] \to k$, the covariant functor $F\colon \C \to \mathrm{Sets}$ associates to any discrete artinian local $W(k)$-algebra $A$, the set of $A$-valued determinants $\D\colon A[G] \to A$, such that $\D\otimes_A k =\bar{\D}$. 

\begin{prop}{~\cite[Proposition E]{1}}
    $F$ is prorepresentable by a complete local noetherian $W(k)$-algebra, which is constructed as a certain profinite completion of $\Gamma_{\Z}^d(\Z[G])^{ab}\otimes_{\Z}W(k)$.
\end{prop} 
We denote it by the pseudodeformation ring $\rps$, and by construction, it is topologically generated by $\Lambda_i(g)$ for $g\in G$ and $i\ge 1$.

%Let $D_{\bar{P}}^{\ps}\colon \C \to \text{Sets}$ be the functor such that, for Artinian local $W(k)$-algebras $(A,\m_A)$ with a fixed indentification $A/\m_A\cong k$. $D_{\bar{P}}^{\ps}$ is the set of pseudorepresentations $P=(T,D)$ valued in $A$ whose mod $\m_A$ reduction is $\bar{P}$. From Proposition E of ~\cite{1} we know $D_{\bar{P}}^{\ps}$ is pro-represented by a complete Noetherian $W(k)$-algebra $R_{\bar{P}}^{\ps}$.

\subsection{Determinants of degree $d=2$}\label{ring}
Given a determinant $\D \colon A[G] \to A$, the corresponding characteristic polynomials $P(\D,m) \in B[X]$ for $m \in B[G]$ have degree $2$ and could be written in the form
$$P(m)\colonequals P(\D,m)=X^2-T(m)X+D(m)$$
for maps $T,D\colon B[G]\to B$. 

If the degree is $2$ and the residue characteristic is different from $2$, we can recover $D$ from $T$ via the identity $$D(\sigma)=\frac{T(\sigma)^2-T(\sigma^2)}{2}.$$
We can also recover $T$ from $D$ via the formula $$T(\sigma)=D(\sigma+1)-D(\sigma)-1.$$
We could regard a determinant $\D$ of $G$ over $A$ of degree $2$ as precisely given by a pair of functions $(T,D)$ satisfying the equations below.

\begin{lemma}{~\cite[Lemma~7.7]{1}}\label{TraceEq}
The above map $\D \mapsto(T, D)$ induces a bijection between the set of $A$-valued determinants of $G$ of dimension $2$ and the set of pairs of functions $(T, D) \colon G \rightarrow A$ such that $D\colon G \rightarrow A^*$ is a group homomorphism, $T\colon G \rightarrow A$ is a function with $T(1)=2$, and which satisfy for all $g, h \in G$:
\begin{enumerate}
    \item $T(g h)=T(h g)$,
    \item $D(g) T\left(g^{-1} h\right)-T(g) T(h)+T(g h)=0$.
\end{enumerate}
\end{lemma}

Given $g \in G$, we have a corresponding characteristic polynomial $P(g)=X^2-T(g)X+D(g)$. By ~\cite[Lemma~7.7]{1}, the functions $T$ and $D$ extends to functions from $A[G]$ to $A$. In the case of $T$, this extension is the linear extension, and in the case of $D$, it can be constructed explicitly by using the equation for $D(gX+hY)$, $g,h \in G$, given by ~\cite[Example~7.6]{1} $$D(gX+hY)\colonequals D(g)X^2+(T(g)T(h)-T(gh))XY+D(h)Y^2.$$
If $A$ is an algebraically closed field, then $(T,D)$ may be realized as the trace and determinant (in the classic definition) of an actual semisimple representation (Theorem A of ~\cite{1}).

By abuse of notation, now we also denote the pair $(T,D)$ by $P=(T,D)$. Let $\bar{P}=(\bar{T},\bar{D}) \colon G \to k$ be a degree $2$ pseudorepresentation.

\subsection{Pseudodeformation Rings of Finite Groups}
In this paper, we will only consider the deformation rings of two-dimensional trivial residual (pseudo)representations of a finite $p$-group $G$ over $k=\F_p$. Denote by $\p\colon G\to \text{GL}_2(\F_p)$ and $\bar{P}\colonequals(\text{tr}(\p),\text{det}(\p))$. Now by our assumptions, for any $g\in G, \bar{P}(g)=(\text{tr}(\p),\text{det}(\p))(g)=(2,1)$. 

\begin{rmk}
    We make this assumption because there will be a unique residual pseudorepresentation of $G$ mod $p$. The only irreducible representation of a $p$-group in character $p$ is the trivial representation (\cite[Proposition~26, Section 8.3]{18}). And in algebraically closed fields, pseudorepresentations always arise from genuine representations. 
\end{rmk}

%\begin{rmk}
%From ~\cite[Lemma~7.64]{1}, the continuous determinant $\textbf{D}\colon A[G]\to A$ deforming $\bar{P}$ factors through $A[G]\to A[G/H]$, where $H \subset \text{ker}(\p)$ is the smallest closed normal subgroup such that $\text{ker}(\p)/H$ is a pro-$p$ quotient of G. So from now on we will only consider a finite $p$-group $G$. 
%\end{rmk}

\begin{prop}
$\rps$ is a finite $\Z_p$-algebra. 
\end{prop}
\begin{proof}
As shown in ~\cite[Proposition~7.4]{1}, ~\cite[Example~7.5(2)]{1} and~\cite[Proposition~7.56]{1}. Another proof is in ~\cite[Proposition~2.13]{4}.  
\end{proof}
  
Therefore we could regard $\rps$ as a finite $\Z_p$-module. 
\begin{defi}
We define a function $d(M)\coloneqq\text{dim}_{\F_p}(M\otimes_{\Z_p}\F_p)$ for a finitely generated $\Z_p$-module $M$. By Nakayama's lemma, we not only count the $\Z_p$-rank of the torsion-free part of $M$, but the generators of the torsion part. 
\end{defi}

From the surjection $\rps\to\rtr$, we will deduce that $d(\rps)\ge d(\rtr)$. Once we show they are free modules with the same rank, we can conclude that they are isomorphic. This is the strategy we use in the case of abelian groups and dihedral groups (see Section 4 for details).

\subsection{An Upper Bound of $d(\rps)$}

\begin{prop}\label{Span}
For any finite group $G$, we could find a generating set $B$ of the $\Z_p$-module $\rps$ satisfying $B \subset L:= \{T_i$,\ $D_i$,\ $T_i T_j\}_{i\ne j}$, where $T_i$ (resp.$D_i$) denote the trace (resp.determinant) of the elements $g_i \in G$.
\end{prop} 
\begin{proof}
We apply the Lemma \ref{TraceEq} to deduce the linear dependence between certain elements.
\begin{enumerate}
    \item $T_i D_j \in$ Span$\{L\}$, since $T_i D_j = T_s T_t - T_l$ for some $s,t,l$.
    \item $T_i T_j D_k \in$ Span$\{L\}$, since $T_i T_j D_k = T_l D_k - D_k D_t T_s $ for some $s,t,l$.
    \item $T_i T_j T_k \in$ Span$\{L\}$, since $T_i T_j T_k = D_l T_s T_k + T_t T_k$ for some $s,t,l$.
    \item $T_i^3 \in$ Span$\{L\}$ because of (3) and $T_i^2 \in$ Span$\{L\}$ since $T(g)^2=T(g^2)+2D(g)$.
    \item For terms of higher degrees, the product of determinants is still the determinant of certain element in the group $G$. And the product of at least $3$ traces could always be reduced to lower degrees because of (2) and (3).
    \item Suppose $g$ is in $[G,G]$, then $D(g)=1$, and for any $h \in G$, $T(g)T(h) = D(g)T(g^{-1}h)+T(gh)=T(g^{-1}h)+T(gh)$. This means $T(g)T(h) \in$  Span$\{L\}$.
\end{enumerate}
\end{proof}

\subsection{A Lower Bound of $d(\rtr)$}
Let $n_K\colonequals$ the number of all the two-dimensional semi-simple representations of $G$ over a splitting $p$-adic field $K$ of characteristic 0. By a splitting field $K$ of $G$, it means any linear representation $\rho_L$ of $G$ over an extension $L$ of $K$ is equivalent to a linear representation $\rho_K$ over $K$. Then over any splitting field of $G$, actually the number $n_K$ is a constant $n_G$ by definition. ~\cite[Proposition~5.1.12]{7} ensures that we could find a finite extension $K$ of $\Q_p$ that is a splitting field for $G$. Now for a fixed splitting field $K$, we label all the semi-simple two-dimensional representations as $\rho_1$, ..., $\rho_{n_G}$. Denote the integer ring of $K$ by $\mathcal{O}_K$, and the maximal ideal of $\mathcal{O}_K$ by $\m_K$. Fix a uniformizer $\pi_K$ of $K$.

\begin{prop}
There exists a splitting field $K$ and a subring $\mathcal{O}$ of $\mathcal{O}_K$, such that up to conjugation, the image of $\rho_i$ is contained in $\GL_2(\mathcal{O})$ and $\rho_i \colon G \to \GL_2(\mathcal{O})$ is a lifting of the trivial residual representation.
\end{prop}

\begin{proof}

Given any $\rho \colon G \to \GL_2(K)$, we could conjugate the image of $\rho$ (which is compact) into $\GL_2(\mathcal{O}_K)$ (the maximal compact subgroup).

Suppose $\mathcal{O}_K/\m_K \cong \F_q$. Now consider $\p \colon G \to \GL(\F_q)$ defined as  $\p \coloneqq \rho \text{ mod } \m_K$. $\text{GL}_2(\F_q)$ has the order $(q^2-1)(q-1)q$. Then its Sylow-$p$ group has order $q$ and is conjugate to the group $H=\{\left(\begin{smallmatrix}1&*\\0&1
\end{smallmatrix}\right)\}$. Then the image of $\p$ of the finite $p$-group $G$ would be contained in a Sylow-$p$ group of $\text{GL}_2(\F_q)$. Therefore $\p$ is unipotent.

If $\p$ is nontrivial, consider $L=K(\sqrt{\pi_K})$ and the scalar extension of $\rho$ into $L$. Fix the uniformizer $\pi_L \coloneqq \sqrt{\pi_K}$ of $L$, then $\rho(g)= \left(\begin{smallmatrix}1+\pi_L^2a(g)& b(g)\\ \pi_L^2c(g)&1+\pi_L^2d(g)
\end{smallmatrix}\right)$. Conjugate it by $\left(\begin{smallmatrix}\pi_L&0\\0&1
\end{smallmatrix}\right)$ and get $\left(\begin{smallmatrix}1+\pi_L^2a(g)& \pi_Lb(g)\\\pi_Lc(g)&1+\pi_L^2d(g)
\end{smallmatrix}\right)$, which is trivial mod $\pi_L$.

By definition $L$ is also a splitting field of $G$. Take $\mathcal{O}\coloneqq $ the subring of $\mathcal{O}_L$ containing $\pi_L$ with the residue field $\F_p$.

\end{proof}

Now we could give the lower bound of $d(\rtr)$.
\begin{prop}\label{Rank}
The $\Z_p$-rank of the torsion-free part of $\rtr$ (as a $\Z_p$-module) is at least $n_G$.
\end{prop}
\begin{proof}
Consider $A=\mathop{\prod}\limits_{i=1}^{n_G} \mathcal{O}$. Now we construct a representation $\rho \colon G \to \text{GL}_2(A)$ by letting $$[\rho(g)]_{ij}=([\rho_1(g)]_{ij},...,[\rho_{n_G}(g)]_{ij}), 1 \le i,j \le 2$$
where $M_{ij}$ denote the entry in the $i$-th row and $j$-th column of the matrix $M$. 

From Example~\ref{Example}, we know $R^{\square}_{\p}=\Z_p\llbracket\{x_g^{ij}\}_{g \in G,1\le i,j\le 2}\rrbracket/I$. We then have a ring map $R^{\square}_{\p} \to A$ taking $x_g^{ij}$ to $[\rho(g)]_{ij}$ for $\{i,j\}=\{1,2\}$ or $\{2,1\}$, and taking $x_g^{ij}$ to $[\rho(g)]_{ij}-1$ for $\{i,j\}=\{1,1\}$ or $\{2,2\}$. $\rtr \to A^{\coeff}$ is a surjection as all the generators of $A^{\coeff}$ (the traces and determinants of $\rho(g)$) are in the image. 

After tensoring with $K$, we have the surjection $\rtr\otimes_{\Z_p} K \to A^{\coeff}\otimes_{\Z_p} K$. From well-known results, $\text{tr}(\rho_i)$ are linearly-independent, so $A^{\coeff}\otimes_{\Z_p} K = K^{n_G}$. Therefore as a $K$-module, the rank of $\rtr\otimes_{\Z_p} K$ is at least $n_G$. And the $K$-rank of $\rtr\otimes_{\Z_p} K$ equals to the $\Z_p$-rank of the torsion-free part of $\rtr$, . 
\end{proof}

\section{An Example Where Deformation Rings Are Different: Extraspecial Group}
We consider the extraspecial group $G=2^{1+4}_+$ of order 32, with the prime $p=2$, the degree $d=2$. The explicit presentation of this group is:
$$G =\langle a,b,c,d | a^4=b^2=d^2=1, c^2=a^2, bab=a^{-1}, ac=ca, ad=da, bc=cb, bd=db, dcd=a^2c \rangle$$
$G$ is a quotient of $D_8 \oplus D_8 = \langle a,b|a^4=b^2=1, bab=a^{-1} \rangle \oplus \langle c,d|c^4=d^2=1, dcd=c^{-1}\rangle$ (modulo by the relation $a^2=c^2$).
\begin{thm}\label{extraspecial} In this case,
$\rps \ncong \rtr$.
\end{thm}

We prove it by constructing a nonzero element in $\rps$ which maps to $0$ in $\rtr$.

\begin{rmk}
$G$ has the property that the commutator group $[G,G]=\{1,a^2\}=\{1,c^2\}$ is isomorphic to $\Z/2\Z$ and $G^{\mathrm{ab}}=G/[G,G] \cong (\Z/2\Z)^4$. So over a splitting field of characteristic 0, all the irreducible finite dimensional representations of the group $G$ consist of 16 irreducible one-dimensional representations and 1 irreducible four-dimensional representations. It follows that over fields of characteristic $0$, $a^2$ is in the intersection of $\mathrm{Ker} \rho$ for all two-dimensional representations $\rho$, that is, for any two-dimensional representations, the trace $\mathrm{tr}(a^2)$ is always $2$.

So if we pass to the generic fiber of $R^{\square}$, $\mathrm{tr}(a^2)-2$ vanishes. It is in the torsion part. A natural question arises as whether $\mathrm{tr}(a^2)-2$ always vanishes in every lifting of the trivial residual representation.
\end{rmk}

\begin{lemma}
$T^{\univ}(a^2)-2$ is nonzero in $\rps$.
\end{lemma}
\begin{proof}
We claim $T^{\univ}(a^2)-2$ is nonzero in $\rps/(2)$, the pseudodeformation ring over $\F_2$. $\rps/(2)$ is a finite ring, so we can use Magma to evaluate it. In Magma, we first construct a multivariate polynomial ring $R$ over $\F_2$ with variables $\{T(g)-2,D(g)-1\}$ for all $g \in G$. Then we define an ideal $I$ containing all the relations of pseudorepresentations as in Lemma \ref{TraceEq}. Let $\m_R$ be the maximal ideal generated by all variables and after completion we will get $\rps = \underset{k\to\infty}{\varprojlim} R/(\m_R^k+I)$. Magma shows that $T(a^2)-2 \notin \m_R^3+I$, therefore it is nonzero in $\rps/(\m^3_{\rps},2)$. Detailed codes could be found in the Appendix.
\end{proof}

\begin{lemma}
$\mathrm{tr}(\rho^{\univ}(a^2))-2$ is zero in $\rtr$.
\end{lemma}
\begin{proof}
We first construct the lifting ring as in the Example \ref{Example}. Here we start from $R = \Z[X_i]/I$ with $X_i$ the entries of $\rho^{\univ}(g)$ of the generators of $G$, and $I$ is the corresponding equations from the relations of the generators. We will show $\text{tr}(\rho^{\univ}(a^2))-2$ is $0$ in $\Z[X_i]/I$, thus it is $0$ in $R^{\square}=\Z_2[[X_i]]/I$, which is the completion of the previous ring. Then it must be $0$ in the subring $\rtr$ of $R^{\square}$. Detailed codes could be found in the Appendix.
\end{proof}
\begin{rmk}
As discussed in ~\cite[Theorem~2.20]{4}, the kernel and cokernel of the ring map from $\rps$ to the GIT quotient ring of $R^{\square}$ (invariants under the adjoint action of $\GL_2$), consists of finite modules of $p$-torsion nilpotents. By the discussion above, the element $\mathrm{tr}(\rho^{\univ}(a^2))-2$ is also zero in the GIT ring, so the kernel of the ring map is nontrivial, containing the element $T^{\univ}(a^2)-2$.
\end{rmk}
\begin{rmk}
For a complete local $k$-algebra $(R,\m)$ with residue field $k$, let 
$$H_R(t)=\sum_{n=0}^{\infty}\text{dim}_k(\m^n/\m^{n+1})t^n\in \Z[[t]]$$
denote the corresponding Hilbert series. Then we could compute in Magma that
$$H_{\rps/(2)}(t)=1+19t+52t^2+49t^3+16t^4.$$
Detailed codes could be found in the Appendix.

To get the Hilbert series of $H_{\rtr/(2)}$, we consider the deformation rings for the group $H \coloneqq (\Z/2\Z)^4 \cong G/[G,G]$, denoted by $\rps_H$ and $\rtr_H$. As proved in the following section, these two deformation rings $H$ are isomorphic. And we have the following diagram:
\begin{center}  
\begin{tikzcd}
\rps\arrow[->>]{r} \arrow[->>]{d} & \rtr \arrow[->>]{d} \\
 \rps_{H} \arrow["\sim"]{r} & \rtr_H
\end{tikzcd}
\end{center}
Let $R_H$ denote the deformation rings $\rps_H$ or $\rtr_H$. From the diagram we have 
$$\rps \twoheadrightarrow \rtr \twoheadrightarrow R_H.$$
We could computed in Magma that $\rps/(2)$ has dimension $137$  and $R_H/(2)$ has dimension $136$ as $\F_2$-vector spaces. Therefore $\rtr/(2)$ is isomorphic to either  $\rps/(2)$ or $R_H/(2)$. In this section we already proved that $T(a^2)-2$ is nontrivial in $\rps/(2)$ but it is trivial in $\rtr$. So $\rtr/(2) \cong R_H/(2)$ and we get the Hilbert series 
$$H_{\rtr/(2)}(t)=H_{R_H/(2)}=1+19t+51t^2+49t^3+16t^4.$$
in Magma.
\end{rmk}

\section{Examples Where Deformation Rings Are Isomorphic: Abelian Groups and Dihedral Groups}
\subsection{Abelian Groups}
\begin{thm}
If $G$ is an abelian $2$-group, then $\rps \cong \rtr \cong \Z_2^r$ for some positive integer $r$.
\end{thm}

\begin{proof}
We will compare the upper bound of $d(\rps)$ and the lower bound of the rank of the torsion-free part of $\rtr$ and show they are the same (and deduce the two deformation rings are both free $\Z_2$-modules).

Denote the cardinality of the abelian group $G$ as $l$. Therefore there are $l$ different conjugacy classes of $G$. Recall in Prop.~\ref{Span} $d(\rps)$ is no larger than the cardinality of the set $L\colonequals \{T_i$,\ $D_i$,\ $T_i T_j\}_{i\ne j}$. Notice that $T(1)=2$, $D(1)=1$. We count the nontrivial conjugacy classes first, and get $(l-1)+(l-1)+\frac{1}{2}(l-1)(l-2)$. Finally we add the element $1$ into the basis. Then $$d(\rps) \le 1+(l-1)+(l-1)+\frac{1}{2}(l-1)(l-2)=\frac{1}{2}(l+1)l.$$ 
Recall in Prop.\ref{Rank}, denote the set $$S\colonequals\{\text{two-dimensional semi-simple representations over a splitting field of characteristic }\ 0\},$$we have the $\Z_2\text{-rank of torsion-free part of the\ }\Z_2\text{-module\ } \rtr$ is at least $$|S| = l+\frac{1}{2}l(l-1)=\frac{1}{2}(l+1)l.$$
Recall in the introduction we have shown $\rps\twoheadrightarrow\rtr$, $$d(\rps) \ge d(\rtr) = \mathrm{rank}_{\Z_2}(\text{torsion-free part of }\rtr)+d(\text{torsion part of }\rtr).$$ By the above discussion, both $\rps$ and $\rtr$ are torsion-free and are isomorphic to each other.
\end{proof}

\subsection{Dihedral Groups}
\begin{thm}
If we have a dihedral group $D_{2m}$, $m$ is a power of $2$ ($m \ge 2$), then $\rps \cong \rtr \cong \Z_2^s$ for some positive integer $s$.
\end{thm}
\begin{proof}
Similar to the abelian group case, from Prop.\ref{Rank} we could deduce the rank of torsion-free part of $\rtr$ is no less than the number of all the semi-simple two-dimensional representations. Over a splitting field of characteristic $0$, the dihedral group $D_{2m}$ would have $4$ one-dimensional irreducible representations and $\frac{1}{2}(m-2)$ two-dimensional irreducible representations. Therefore,$$\text{the\ }\Z_2 \text{-rank of torsion-free part of\ }\rtr\ge4+\frac{4\times3}{2}+\frac{m-2}{2} = \frac{m+18}{2}.$$\\
On the other hand, $D_{2m}$ have $\frac{1}{2}(m+6)$ conjugacy classes, and $3$ of them are not in $[G,G]$. Put aside the identity element for a while. We have 
\begin{enumerate}
    \item at most $\frac{1}{2}(m+4)$ different $T_i$ which are not $2$,
    \item at most $3$ different $D_i$ which are not $1$,
    \item at most $\frac{1}{2}(3\times2)$ different $T_iT_j$ which are not in the Span$\{T_i\}$.
\end{enumerate}(2) and (3) comes from Prop.\ref{Span}(6). Add the element $1$ back, we have
$$d(\rps)\le 1+\frac{m+4}{2}+3+\frac{3\times2}{2} = \frac{m+18}{2}.$$

Similar to the abelian group case, we could deduce that $\rps$ and $\rtr$ are isomorphic free $\Z_2$-modules.
\end{proof}

\section*{Acknowledgement}
I would like to thank my advisor Frank Calegari for introducing me to this area, and for his constant support and helpful discussions on this topic. I would also like to thank Professor Carl Wang-Erickson for taking the time to read an earlier draft of this paper and for his valuable comments and corrections. I also want to thank Chengyang Bao and Abhijit Mudigonda for useful discussions.
\appendix
\section{Codes}

\subsection{$\rps$}

We use Magma to get the pseudodeformation ring over $\F_2$.
~\\
\begin{Verbatim}
>R:=PolynomialRing(GF(2),64);
>G:=SmallGroup(32,49);//the GAP ID of the extraspecial group of order 32
>num_map:=NumberingMap(G);

>I:=ideal<R|R.1,R.33>;
>m:=ideal<R|0>;
>// We let R.i denote T(g_i), and R.(i+32) denote D(g_i)-1,
>// and require T(1)=0, D(1)=1 in characteristic 2

>for x in [1..64]
>   do m:=m+ideal<R|R.x>;
>end for;// generate the maximal ideal in the polynomial ring

>for l,n in [1..32]
>	do for k in [1..32] 
>	    do if (k @@ num_map) eq (l @@ num_map)*(n @@ num_map) then 
>        	I:=I+ideal<R| R.(k+32)-R.(l+32)*R.(n+32) -R.(l+32)-R.(n+32)>;
>		for a in [1..32] 
>   		do if (a @@ num_map) eq (n @@ num_map)*(l @@ num_map) then
>       		I:=I+ideal<R|R.k-R.a>;
>   		end if;
>		end for;
>		for j in [1..32]
>   		do if (j @@ num_map) eq (l @@ num_map)/(n @@ num_map) then
>                       I:=I+ideal<R|R.(n+32)*R.j+R.j-R.l*R.n+R.k>;		
>   		end if;
>    	    end for;
>   	end if;
>	end for;
>end for;
>//Add all the relations for pseudo representations into the ideal I

>// R.2 is the special element tr(a^2) we want to check
> R.2 in m+I;
true
> R.2 in m^2+I;
true
> R.2 in m^3+I;
false
>// R.2 is nonzero in the deformation ring since it is not in m^3+I

 
>Dimension(quo<R|m+I>);
1
>Dimension(quo<R|m^2+I>);
20
>Dimension(quo<R|m^3+I>);
72
>Dimension(quo<R|m^4+I>);
121
>Dimension(quo<R|m^5+I>);
137
>m^5 subset I;
true
>Dimension(quo<R|I>);
137
>//Compute the Hilbert Series
\end{Verbatim}

\subsection{$\rtr$}

 We use Sage instead of Magma to evaluate the special element since Sage runs faster in this computation.

\begin{lstlisting}
sage: R.<x1,x2,x3,x4,y1,y2,y3,y4,z1,z2,z3,z4,w1,w2,w3,w4>=ZZ[]

# as explained in Lemma 3.4, we start from the deformation ring over the integer ring
# x_i are entries of the matrix form of the universal representation of the generator a
# similarly y_i for b, z_i for c, w_i for d

sage: J1=ideal(4*x1^3 + x1^4 + 3*x1^2*
    (2 + x2*x3) + x2*x3*(6 + x2*x3 + 
     4*x4 + x4^2) + 
   2*x1*(2 + x2*x3*(4 + x4)), 
  x2*(2 + x1 + x4)*(2 + 2*x1 + x1^2 + 
    2*x2*x3 + 2*x4 + x4^2),x3*(2 + x1 + x4)*(2 + 2*x1 + x1^2 + 
    2*x2*x3 + 2*x4 + x4^2), 
  x2^2*x3^2 + x4*(4 + 6*x4 + 4*x4^2 + 
     x4^3) + x2*x3*(6 + x1^2 + 8*x4 + 
     3*x4^2 + 2*x1*(2 + x4)),2*y1 + y1^2 + y2*y3, 
  y2*(2 + y1 + y4),y3*(2 + y1 + y4), y2*y3 + 
   y4*(2 + y4),2*w1 + w1^2 + w2*w3, 
  w2*(2 + w1 + w4),w3*(2 + w1 + w4), w2*w3 + 
   w4*(2 + w4))
  
# adding relations a^4=b^2=d^2=1 into the ideal

sage: J2=ideal((1 + x1)^2 + x2*x3 - (1 + z1)^2 - z2*z3, x2*(2 + x1 + x4) - 
   z2*(2 + z1 + z4),x3*(2 + x1 + x4) - 
   z3*(2 + z1 + z4), 
  x2*x3 + (1 + x4)^2 - z2*z3 - 
   (1 + z4)^2,-1 + (1 + x1)*
    ((1 + y1)*(1 + x1 + y1 + x1*y1 + 
       x3*y2) + (x2 + x2*y1 + y2 + 
       x4*y2)*y3) + 
   x3*(x2*(1 + y1)*(1 + y4) + 
     y2*(2 + x1 + x4 + y1 + x1*y1 + 
       x3*y2 + y4 + x4*y4)), 
  (x2 + x2*y1 + y2 + x4*y2)*
   (2 + x1 + x4 + y1 + x1*y1 + 
    x3*y2 + x2*y3 + y4 + x4*y4),(x3 + y3 + x1*y3 + x3*y4)*
   (2 + x1 + x4 + y1 + x1*y1 + 
    x3*y2 + x2*y3 + y4 + x4*y4), 
  -1 + (1 + x4)*
    (y2*(x3 + y3 + x1*y3 + x3*y4) + 
     (1 + y4)*(1 + x4 + x2*y3 + y4 + 
       x4*y4)) + 
   x2*(x3*(1 + y1)*(1 + y4) + 
     y3*(2 + x1 + x4 + y1 + x1*y1 + 
       x2*y3 + y4 + x4*y4)))

# relations c^2=a^2, baba=1

sage: J3=ideal(-(x3*z2) + x2*z3, -(x2*z1) + 
   x1*z2 - x4*z2 + x2*z4,x3*z1 - x1*z3 + x4*z3 - x3*z4, 
  x3*z2 - x2*z3,w3*x2 - w2*x3, w2*x1 - w1*x2 + 
   w4*x2 - w2*x4,-(w3*x1) + w1*x3 - w4*x3 + w3*x4, 
  -(w3*x2) + w2*x3,-(y3*z2) + y2*z3, -(y2*z1) + 
   y1*z2 - y4*z2 + y2*z4,y3*z1 - y1*z3 + y4*z3 - y3*z4, 
  y3*z2 - y2*z3,w3*y2 - w2*y3, w2*y1 - w1*y2 + 
   w4*y2 - w2*y4,-(w3*y1) + w1*y3 - w4*y3 + w3*y4, 
  -(w3*y2) + w2*y3,2*w1 + w1^2 + w2*w3 - 2*x1 - x1^2 - 
   x2*x3 + 2*w1*z1 + w1^2*z1 - 
   2*x1*z1 - x1^2*z1 - x2*x3*z1 + 
   w3*z2 + w1*w3*z2 + w2*z3 + 
   w1*w2*z3 - 2*x2*z3 - x1*x2*z3 - 
   x2*x4*z3 + w2*w3*z4, 
  2*w2 + w1*w2 + w2*w4 - 2*x2 - 
   x1*x2 - x2*x4 + w2*z1 + w1*w2*z1 + 
   w1*z2 + w4*z2 + w1*w4*z2 - 
   2*x1*z2 - x1^2*z2 - x2*x3*z2 + 
   w2^2*z3 + w2*z4 + w2*w4*z4 - 
   2*x2*z4 - x1*x2*z4 - x2*x4*z4,2*w3 + w1*w3 + w3*w4 - 2*x3 - 
   x1*x3 - x3*x4 + w3*z1 + w1*w3*z1 - 
   2*x3*z1 - x1*x3*z1 - x3*x4*z1 + 
   w3^2*z2 + w1*z3 + w4*z3 + 
   w1*w4*z3 - x2*x3*z3 - 2*x4*z3 - 
   x4^2*z3 + w3*z4 + w3*w4*z4, 
  w2*w3 + 2*w4 + w4^2 - x2*x3 - 
   2*x4 - x4^2 + w2*w3*z1 + w3*z2 + 
   w3*w4*z2 - 2*x3*z2 - x1*x3*z2 - 
   x3*x4*z2 + w2*z3 + w2*w4*z3 + 
   2*w4*z4 + w4^2*z4 - x2*x3*z4 - 
   2*x4*z4 - x4^2*z4)
 

# relations ac=ca, ad=da, bc=cb, bd=db, dcd=a^2c
  
sage: I = J1+J2+J3
  
sage: x1^2+x4^2+2*x1+2*x4+2*x2*x3 in I
sage: True
# tr(a^2)-2 is in I

sage:z1^2+z4^2+2*z1+2*z4+2*z2*z3 in I
sage: True
# double check: tr(c^2)-2 is in I

\end{lstlisting}


\begin{thebibliography}{WWE17}
\providecommand{\url}[1]{\texttt{#1}}
\providecommand{\urlprefix}{URL }
\providecommand{\eprint}[2][]{\url{#2}}

%18
\bibitem[BCP97]{13}
Wieb Bosma, John Cannon, and Catherine Playoust, \textit{The Magma algebra system}. I. The user language, J. Symbolic Comput. \textbf{24} (1997), no. 3-4, p. 235–265, Computational algebra and number theory (London, 1993). MR 1484478
\bibitem[BG09]{8}
Joël Bellaïche and Gaëtan Chenevier, \textit{Families of Galois representations and Selmer groups}. Astérisque, no. 324 (2009), 326 p. \url{http://numdam.org/item/AST_2009__324__R1_0/}.
\bibitem[Che14]{1}
Gaëtan Chenevier, \textit{The p-adic analytic space of pseudocharacters of a profinite group and pseudorepresentations over arbitrary rings}, Automorphic forms and Galois representations. Vol. 1, London Math. Soc. Lecture Note Ser., vol. 414, Cambridge Univ. Press, Cambridge, 2014, p. 221-285. MR 3444227
\bibitem[CS19]{11}
Frank Calegari and Joel Specter, \textit{Pseudo-representations of weight one are unramified}. Algebra Number Theory 13 (7) 1583 - 1596, 2019. \url{https://doi.org/10.2140/ant.2019.13.1583}.
\bibitem[Fro]{12}
Ferdinand Georg Frobenius, \textit{Über die Primfactoren der Gruppendeterminante}, Ges. Abh. III (1968) (S'ber. Akad. Wiss. Berlin 1343–1382).
\bibitem[Gro97]{7}
Larry C. Grove, \textit{Groups and Characters}, Wiley(1997).
\bibitem[Maz89]{6}
Barry Mazur, \textit{Deforming Galois representations}, Galois groups over $\mathbf{Q}$ (Berkeley, CA, 1987), Math.
Sci. Res. Inst. Publ., vol. 16, Springer, New York, 1989, p. 385-437. MR 1012172
\bibitem[Ny96]{10}
Louise Nyssen, \textit{Pseudo-représentations}, Math. Ann. 306 (1996), p. 257-283.
\bibitem[Pro74]{15} Claudio Procesi, \textit{Finite dimensional representations of algebras}, Israel Journal of Mathematics 19(1974), p. 169–182.
\bibitem[Pro76]{16}  Claudio Procesi, \textit{Invariant theory of n × n matrices}, Advances in Math 19(1976), p. 306–381.
\bibitem[Pro87]{17} Claudio Procesi, \textit{A formal inverse to the Cayley Hamilton theorem}, Journal of Algebra 107(1987), p. 63–74.
\bibitem[Rou96]{9} 
Raphaël Rouquier, \textit{Caractérisation des caractères et pseudo-caractères}, J. Algebra 180 (1996), p. 571-586.
\bibitem[Sage20]{14}
The Sage Developers, \textit{Sagemath, the Sage Mathematics Software System (Version 9.1)}, 2020, \url{https://www.sagemath.org}.
\bibitem[Ser77]{18}
Jean-Pierre Serre, \textit{Linear representations of finite groups}, Graduate Texts in Mathematics Springer (1977).
\bibitem[Tay91]{3}
Richard Taylor, \textit{Galois representations associated to Siegel modular forms of low weight}, Duke Math. J. 63 (1991), no. 2, p. 281-332. MR 1115109
\bibitem[WE18]{4}
Carl Wang-Erickson, \textit{Algebraic families of Galois representations and potentially semi-stable pseudodeformation rings}, Math. Ann. 371 (2018), no. 3-4, p. 1615-1681. MR 3831282
\bibitem[WWE17]{5}
Preston Wake and Carl Wang-Erickson, \textit{Ordinary pseudorepresentations and modular forms}, Proc. Amer. Math. Soc. Ser. B 4 (2017), p. 53-71. MR 3738092
\bibitem[Wil88]{2}
Andrew Wiles, \textit{On ordinary $\lambda$-adic representations associated to modular forms}, Invent. Math. 94 (1988), no. 3, p. 529-573. MR 969243

\end{thebibliography}
\end{document}